\documentclass[fleqn
,12pt
]{article}
\usepackage{amsmath,amssymb,amsthm,xcolor,framed,tikz}
\usepackage[english]{babel}
\usepackage[margin=3.5cm
,top=3cm,bottom=3cm
]{geometry}
\usepackage[title]{appendix}
\usepackage[bookmarks]{hyperref}
\numberwithin{equation}{section}

\def\XXint#1#2#3{{\setbox0=\hbox{$#1{#2#3}{\int}$}
		\vcenter{\hbox{$#2#3$}}\kern-.5\wd0}}

\newcommand{\R}{{\mathbb R}}

\renewcommand{\S}{{\mathbb S}}

\newcommand{\e}{\epsilon}
\DeclareMathOperator{\dv}{div}
\DeclareMathOperator{\curl}{curl}

\newcommand{\loc}{{\rm loc}}
\newtheorem{thm}{Theorem}[section]

\newtheorem{conj}[thm]{Conjecture}
\newtheorem{lem}[thm]{Lemma}

\theoremstyle{definition}
\newtheorem{rem}[thm]{Remark}
\newtheorem*{rem*}{Remark}

\title{On the minimality of the radial singularity in nematic liquid crystals}
\date{\today}
\author{Pierre Bousquet\thanks{Univ Toulouse, CNRS, IMT, Toulouse, France, pierre.bousquet@math.univ-toulouse.fr} 
\and Xavier Lamy\thanks{Institut Universitaire de France (IUF) \& Univ Toulouse, CNRS, IMT, Toulouse, France,
xavier.lamy@math.cnrs.fr}
}

\begin{document}

\maketitle

\begin{abstract}
We show that the map $u_*\colon B_1\subset\mathbb R^3\to\mathbb S^2$, $x\mapsto x/|x|$, minimizes the anisotropic energy 
\begin{align*}
\int_{B_1} 
\Big(
k_1(\mathrm{div}\, u)^2+k_2 (u\cdot \mathrm{curl}\, u)^2 +k_3 |u\times\mathrm{curl}\, u|^2
\Big)
\, dx,
\end{align*}
among $\mathbb S^2$-valued maps agreeing with it on the boundary,
for values of $k_1,k_2,k_3>0$ beyond the known regime $k_1\leq k_2$.
Our main tool is a new stability estimate for $u_*$ as a
minimizing sphere-valued harmonic map.
\end{abstract}

\section{Introduction}

We are interested in the Oseen-Frank energy
\begin{align}\label{eq:Ek}
E_{\mathbf k}(u)=
\int_{B_1} 
\Big(
k_1(\dv u)^2+k_2 (u\cdot \curl u)^2 +k_3 |u\times\curl u|^2
\Big)
\, dx,
\end{align}
of maps $u\colon B_1\subset\R^3\to\mathbb S^2\subset\R^3$ which agree with the radial map
\begin{align}\label{eq:u*}
u_*(x)=\frac{x}{|x|}\,,
\end{align} 
on the boundary $\partial B_1$.
Here $k_1,k_2,k_3 >0$ are the elastic constants associated to splay, twist and bend deformations in a nematic liquid crystal \cite{degennes},
 and we denote their triple by $\mathbf k=(k_1,k_2,k_3)$.
The admissible set of maps is defined by
 \begin{align}
 H^1_{u_*}(B_1;\mathbb S^2)
 =\Big\lbrace u\in H^1(B_1;\R^3)\colon 
 &
 u(x)\in\mathbb S^2\text{ for a.e. }x\in B_1
 \nonumber
 \\
 &
 \text{and }u=u_*\text{ on }\partial B_1\Big\rbrace\,.
 \label{eq:H1u*}
 \end{align}
The radial map $u_*$ is a critical point of $E_{\mathbf k}$ for all values of $\mathbf k$, see e.g.~\cite[Proposition 3]{ou92}. However, 
that \(u_*\) minimizes $E_{\mathbf k}$ or not,
 is only known for some, but not all values of 
 $\mathbf k$. Our main results partially fill this gap.
 
  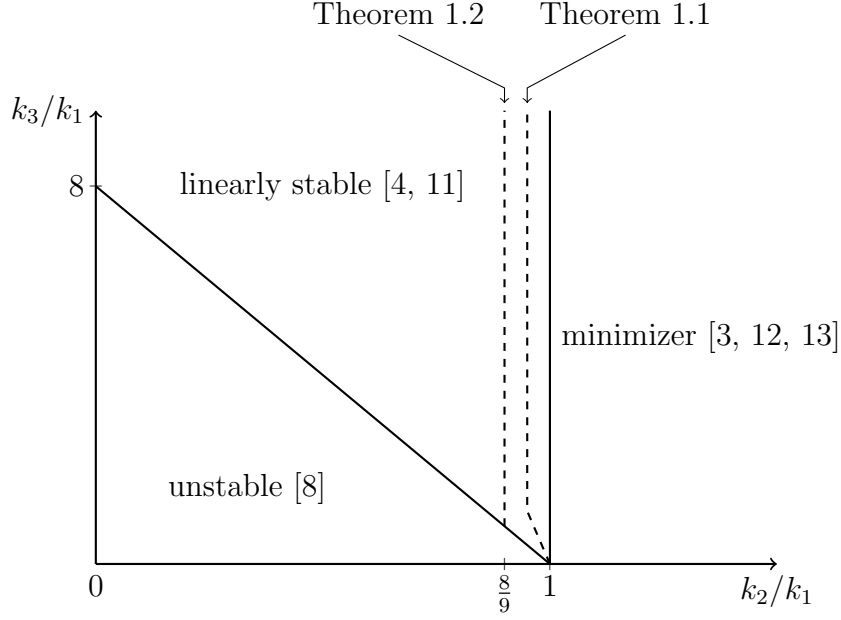
\begin{figure}[h]
\begin{center}
\begin{tikzpicture}[scale=1]

\draw[thick,->] (0,0) node [below] {0} -- (9,0) node [below] {$k_2/k_1$};

\draw (8,3) node {minimizer \cite{BCL86,lin87,ou92}};

\draw (2,1) node {unstable \cite{helein87}};

\draw (3,5) node {linearly stable \cite{CT90,KO92}};

\draw[thick,->] (0,0) -- (0,6) node [left] {$k_3/k_1$};

\draw[thick] (6,0) node {\tiny $+$} node [below] {1} -- (6,6);

\draw (5.4,0) node {\tiny $+$} node [below] {$\frac 89$};

\draw[thick] (6,0) -- (0,5) node {\tiny $+$} node [left] {8};

\draw[thick,dashed] (5.4,0.5) -- (5.4,6);

\draw[thick,dashed] (6,0) -- (5.7,.7) -- (5.7,6);

\draw[->] (4,7) node [above] {Theorem~\ref{t:minOF0}} -- (5.4,6.3) -- (5.4,6.1);

\draw[->] (7,7) node [above] {Theorem~\ref{t:minOF}} -- (5.7,6.3) -- (5.7,6.1);

\end{tikzpicture}
\end{center}
\caption{Known facts about the minimality of $u_*$, and new results proven in this article:
minimality on the right of the dashed lines, among all competitors (Theorem~\ref{t:minOF}) or those only singular at the origin (Theorem~\ref{t:minOF0}).}
\end{figure}

The first known result in that direction \cite{BCL86} is that
 $u_*$ is the unique minimizer of the Dirichlet energy 
 \begin{align}\label{eq:D}
 D(u)=\int_{B_1}|\nabla u|^2\, dx\,.
 \end{align}  
This is relevant to the case of equal elastic constants $\mathbf k=\mathbf 1=(1,1,1)$,
because  $(\dv u)^2+|\curl u|^2-|\nabla u|^2=\dv((\dv u)u-(u\cdot\nabla)u)$ is a null Lagrangian: the quantity \(E_{\mathbf 1}(u)-D(u)\) depends only on the fixed boundary datum \(u=u_*\) on \(\partial B_1\) and thus
\begin{equation}\label{eq171}
E_{\mathbf 1}(u)- D(u)=E_{\mathbf 1}(u_*)- D(u_*)=16\pi-8\pi=8\pi\,.
\end{equation}
 It follows that  $u_*$ is also  the unique minimizer of the Oseen-Frank energy $E_{\mathbf 1}$.

More generally, it is known that $u_*$ is the unique minimizer of $E_{\mathbf k}$ provided that
 $k_2\geq k_1$ \cite{lin87,ou92}.
In the opposite direction,  $u_*$ is not a minimizer of $E_{\mathbf k}$ 
if $k_2 + k_3/8 < k_1$ \cite{helein87}.
 For the remaining values $k_2 < k_1 \leq k_2 +k_3/8$, the radial map  $u_*$ is linearly stable (i.e. the second variation of $E_{\mathbf k}$ at $u_*$ is positive definite) \cite{CT90,KO92}, 
but its minimality  is open.
Numerical simulations in \cite{AG97} 
suggest that it might be non-minimizing for some of these linearly stable values.

Our main result establishes that $u_*$ is 
a minimizer
at least slightly beyond the previously known regime $k_2\geq k_1$.

\begin{thm}\label{t:minOF}
There exists $\delta \in (0,1/8]$ such that,
for any $k_1,k_2,k_3>0$ with $k_1\leq k_2 + \delta \min(k_2, k_3)$,
the radial hedgehog $u_*$ is the unique minimizer of the Oseen-Frank energy $E_{\mathbf k}$ in $H^1_{u_*}(B_1;\mathbb S^2)$.
\end{thm}

The proof of~Theorem~\ref{t:minOF} relies on a compactness argument and we are not able to provide an optimal value of \(\delta\).  
However, if we restrict ourselves to competitors which are continuous away from the origin, 
we can actually show minimality under the condition
 $k_1\leq k_2+\min(k_2,k_3)/8$, 
 which, for $k_2\geq k_3$, corresponds exactly to the regime of linear stability $k_1\leq k_2 +k_3/8$.

\begin{thm}\label{t:minOF0}
For any $k_1,k_2,k_3>0$ with $k_1\leq k_2 + \min(k_2,k_3)/8$,
the radial hedgehog $u_*$
is the unique minimizer of the Oseen-Frank energy $E_{\mathbf k}$ among maps in $H^1_{u_*}(B_1;\mathbb S^2)$ which are continuous in $\overline B_1\setminus\lbrace 0\rbrace$.
\end{thm}

\begin{rem}
As noted in \cite{CT90}, even though the second variation of $E_{\mathbf k}$ at $u_*$ is positive definite for $k_1<k_2+k_3/8$,
and this implies local minimality of $u_*$ with respect to small Lipschitz perturbations,
it is not clear whether this implies local minimality of $u_*$
with respect to the natural $H^1$ distance.
This is because the functional setting is not smooth enough,
as can be seen e.g. by moving the singularity from the origin  \cite[Proposition~5]{CT90}.
In fact, the map $u_*$ is a minimizer if and only if it is a local minimizer (in the $H^1$ distance).
Assume indeed that $u_*$ is not a minimizer. Then, there exists $\tilde u\in H^1_{u_*}(B_1;\S^2)$ such that $E_{\mathbf k}(\tilde u)< E_{\mathbf k}(u_*)$.
Extend $\tilde u$ to $\R^3$ by $u_*$, and define $\tilde u_\e(x)=\tilde u(x/\e)$.  Then, $\tilde u_\e\to u_*$ in $H^1$ as $\e\to 0$, and $E_{\mathbf k}(\tilde u_\e)-E_{\mathbf k}(u_*)=\e (E_{\mathbf k}(\tilde u)-E_{\mathbf k}(u_*))<0$, so $u_*$ is not a local minimizer.
\end{rem}

Both Theorem~\ref{t:minOF} and Theorem~\ref{t:minOF0}
rely on
 stability estimates for $u_*$ as a minimizer of the Dirichlet energy $D(u)=\int_{B_1} |\nabla u|^2\, dx$ among competitors with one singularity,
 which are of independent interest.

\begin{thm}\label{t:stabx0}
There exists $c_*>0$ such that, 
if
$u\in H^1_{u_*}(B_1;\mathbb S^2)\cap C^0(\overline B_1\setminus \lbrace x_0\rbrace)$
for some $x_0\in B_1$, we have
\begin{align}\label{eq:stabx0}
c_* 
\int_{B_1}|\nabla u_*(\cdot -x_0)-\nabla u|^2\, dx 
\leq
D(u)-D(u_*)
\,.
\end{align}
When \(x_0=0\), one can take $c_*=1/9$.
\end{thm}

Even if the above statement is restricted to maps \(u\) having just one singularity, it is sufficient for the proof of Theorem~\ref{t:minOF}, due to the fact that for $\mathbf k\approx \mathbf 1$, minimizers of $E_{\mathbf k}$ have only one singularity \cite{AL88}.
However, we expect that a suitable formulation of~\eqref{eq:stabx0} is still valid on  the whole set \(H^1_{u_*}(B_1;\mathbb S^2)\):

\begin{conj}\label{conj:stab}
There exists $c_*>0$ such that
\begin{align}
\label{eq:stab_conj}
c_* \inf_{x_0\in B_{1}}
\int_{B_1}|\nabla u_*(\cdot -x_0)-\nabla u|^2\, dx 
\leq 
D(u)-D(u_*)
\,,
\end{align}
for all  $u\in H^1_{u_*}(B_1;\mathbb S^2)$.
\end{conj}

Since the map $u_*$ is the unique minimizer of the 
Dirichlet energy
$D$ in  $H^{1}_{u_*}(B_1;\S^2)$,
and the second variation of $D$ at $u_*$ is a positive definite quadratic form \cite{CT90},
one might be tempted to expect 
the nondegenerate, quadratic stability estimate: 
\begin{align}\label{eq:nondeg_stab}
D(u)-D(u_*)\geq c_*\int_{B_1}|\nabla u-\nabla u_*|^2\, dx\,.
\end{align}
But the example
\begin{align*}
u^{x_0}(x)=u_*(x- \eta(|x|)x_0)
\,,
\end{align*}
for a smooth function $\eta$ with $\mathbf 1_{r\leq 1/2}\leq \eta(r)\leq \mathbf 1_{r\leq 3/4}$ and $|x_0|\leq 1/4$,
actually shows that 
\eqref{eq:nondeg_stab} is not possible: one checks by direct calculation that
\begin{align*}
\int_{B_1}|\nabla u^{x_0}-\nabla u_*|^2\, dx \gtrsim
\sqrt{D(u^{x_0})-D(u_*)}\,,
\end{align*}
 see Lemma~\ref{l:ex} below.
The main feature of this example is to translate the singularity,
and Conjecture~\ref{conj:stab} 
asserts that this should be the only source of degeneracy.

\begin{rem}
In higher dimensions $n\geq 4$, the radial map $u_*(x)=x/|x|$ is also a minimizer of the Dirichlet energy among $\mathbb S^{n-1}$-valued maps agreeing with $u_*$ on the boundary of $B_1\subset\R^n$ \cite{lin87}.
For $n\geq 7$ the nondegenerate stability estimate \eqref{eq:nondeg_stab} is known :  see the proof of \cite[Theorem~A~(ii)]{Hong2000} and remark a in that paper.
For $n=4$, translations $u^{x_0}$ also provide a counterexample, see Remark~\ref{r:ux0higherdim}.
The validity of \eqref{eq:nondeg_stab} for $n\in \lbrace 5,6\rbrace$ seems open.
Note that, thanks to  \cite[Theorem~2.1]{helein89} it would be enough to check it on maps $u$ which are continuous in $\overline B_1\setminus \lbrace 0\rbrace$.
\end{rem}

From~\eqref{eq:stabx0}, using that $|\curl v|^2\leq 2|\nabla v|^2$ and  that \(u_*\) as well as its translations are curl-free, one gets
\begin{equation}\label{eq:curl}
c_*\int_{B_1}|\curl  u|^2\,dx \leq D(u)-D(u_*).
\end{equation}
As a matter of fact, this weaker form of ~\eqref{eq:stabx0} is sufficient to establish Theorems~\ref{t:minOF} and \ref{t:minOF0}.

\begin{rem}\label{r:curl}
Another indication supporting Conjecture~\ref{conj:stab} is that 
the nondegenerate  estimate~\eqref{eq:curl} on $\curl u$ is actually valid among all competitors \(u\in H^1_{u_*}(B_1;\mathbb S^2)\) and not just for  those which are continuous on \(\overline{B}_1\) except at one point.
To check this, 
note that, in view of  Theorem~\ref{t:minOF}, there exists $\delta >0$ such that the radial map $u_*$ minimizes $E_{\mathbf k}$, with $\mathbf k=(1+\delta,1,1)$.
Using this minimality and the fact that $\curl u_*=0$, we obtain that for every \(u\in H^{1}_{u_*}(B_1;\mathbb{S}^2)\), 
\begin{align*}
\frac{\delta}{1+\delta}\int_{B_1}|\curl u|^2\, dx 
&
=E_{\mathbf 1}(u) -\frac{1}{1+\delta}E_{\mathbf k}(u)
\\
&
\leq E_{\mathbf 1}(u)-\frac{1}{1+\delta}E_{\mathbf k}(u_*)
=E_{\mathbf 1}(u)-E_{\mathbf 1}(u_*)\,, 
\end{align*}
and we have already seen that the last expression equals
$D(u)-D(u_*)$, see~\eqref{eq171}.
\end{rem}

\subsubsection*{Organization of the article}
We prove the stability estimate of Theorem~\ref{t:stabx0} in \textsection\ref{s:stab} and the minimality statements of Theorem~\ref{t:minOF} and Theorem~\ref{t:minOF0} in \textsection\ref{s:minOF}.

\subsubsection*{Acknowledgments}

XL is supported by the ANR project ANR-22-CE40-0006.
PB and XL warmly thank Rémy Rodiac for very  inspiring discussions all along this project, 
as well as André Guerra for his insightful comments.

\section{Stability among maps with one singularity}\label{s:stab}

In this section we prove
the quadratic stability of $u_*$ modulo translations for the Dirichlet energy among sphere-valued maps which are discontinuous at just one point, as stated in Theorem~\ref{t:stabx0}.

\subsection{Singularity at the origin}

We start by presenting the short proof of Theorem~\ref{t:stabx0} for maps $u\in H^1_{u_*}(B_1;\S^2)$ which are continuous away from the origin $x_0=0$.
This assumption ensures that the restrictions of $u$ to concentric spheres have topological degree equal to one,
hence enables us to use the minimality of $u_*\colon\mathbb S^2\to\mathbb S^2$ in
 its homotopy class.

\begin{lem}\label{l:stab_sing_origin}
If $u\in H^1_{u_*}(B_1;\mathbb S^2)$ is continuous in $\overline B_1\setminus \lbrace 0\rbrace$, then
\begin{align}\label{eq:stab_sing_origin}
\frac 19 \int_{B_1}|\nabla u-\nabla u_*|^2\, dx \leq \int_{B_1}|\nabla u|^2\, dx - 8\pi\,.
\end{align}
Equality occurs if and only if $u=u_*$.
\end{lem}

\begin{proof}
For a.e. $r\in (0,1)$, the map $ u^{(r)} \colon\omega\mapsto u(r\omega)$ belongs to $H^1(\mathbb S^2;\mathbb S^2)$,
 and it has degree one thanks to the continuity assumption.
Classically, the inequality $|\nabla_\omega u^{(r)}|^2\geq 2|\det(\nabla_{\omega}u^{(r)})|$ 
(where $\nabla_\omega u^{(r)}$ is seen as a linear map from $T_\omega\mathbb S^2$ to $T_{u^{(r)}(\omega)}\mathbb S^2$) therefore implies
\begin{align*}
\int_{\mathbb S^2} |\nabla_\omega u^{(r)}|^2\,d\mathcal H^2(\omega)\geq 8\pi |\deg(u^{(r)})|=8\pi \,.
\end{align*}
Thus we have, using also classically that $-\Delta_\omega u_*=2u_*$ and $2u_*\cdot(u_*-u^{(r)})=|u^{(r)}-u_*|^2$,
\begin{align*}
0
&
\leq 
\int_{\mathbb S^2}|\nabla_\omega u^{(r)}|^2\, d\mathcal H^2(\omega) -8\pi
\\
&
=\int_{\mathbb S^2}|\nabla_\omega u^{(r)}|^2\, d\mathcal H^2(\omega)-\int_{\mathbb S^2}|\nabla_\omega u_*|^2\, d\mathcal H^2(\omega)
\\
&
=\int_{\mathbb S^2}|\nabla_\omega u^{(r)}-\nabla_\omega u_*|^2\, d\mathcal H^2(\omega) 
- 2\int_{\mathbb S^2}\nabla_\omega u_*\cdot \nabla_\omega(u_*-u^{(r)})\,d\mathcal H^2(\omega)
\\
&
=\int_{\mathbb S^2}|\nabla_\omega u^{(r)}-\nabla_\omega u_*|^2\, d\mathcal H^2(\omega) 
- 4\int_{\mathbb S^2}u_*\cdot (u_*-u^{(r)})\,d\mathcal H^2(\omega)
\\
&
=\int_{\mathbb S^2}|\nabla_\omega u^{(r)}-\nabla_\omega u_*|^2\, d\mathcal H^2(\omega) 
- 2\int_{\mathbb S^2}|u^{(r)}-u_*|^2\,d\mathcal H^2(\omega)\,.
\end{align*}
For any $\theta\in [0,1]$, we deduce
\begin{align*}
&
\int_{\mathbb S^2}|\nabla_\omega u^{(r)}|^2\, d\mathcal H^2(\omega) -8\pi
\\
&
\geq 
\theta\Bigg(
\int_{\mathbb S^2}|\nabla_\omega u^{(r)}-\nabla_\omega u_*|^2\, d\mathcal H^2(\omega) 
- 2\int_{\mathbb S^2}|u^{(r)}-u_*|^2\,d\mathcal H^2(\omega)\bigg)\,.
\end{align*}
Integrating this with respect to $r\in (0,1)$ and adding the integral of $|\partial_r u|^2$, we find
\begin{align*}
&
\int_{B_1}|\nabla u|^2\, dx - 8\pi
\\
&
\geq \theta\int_{B_1}|\nabla u-\nabla u_*|^2\, dx
+(1-\theta)\int_{B_1}|\partial_r u|^2\, dx - 2\theta\int_{B_1}\frac{|u-u_*|^2}{|x|^2}\, dx\,.
\end{align*}
Further, applying Hardy's inequality
\begin{align*}
\frac 14 \int_0^1 |f(r)|^2\, dr \leq \int_0^1 |f'(r)|^2 r^2 dr\qquad\text{if }f(1)=0\,,
\end{align*}
to $f(r)=f_\omega(r)=u(r\omega)-u_*(r\omega)$ for a.e. $\omega\in\mathbb S^2$,
and integrating with respect to $\omega$, we get
\begin{align*}
\int_{B_1}|\partial_r u|^2\, dx \geq \frac 14 
\int_{B_1}\frac{|u-u_*|^2}{|x|^2}\, dx\,.
\end{align*}
Plugging this into the above estimate yields
\begin{align*}
&
\int_{B_1}|\nabla u|^2\, dx - 8\pi
\\
&
\geq \theta\int_{B_1}|\nabla u-\nabla u_*|^2\, dx
+\bigg(\frac{1-\theta}{4}-2\theta\bigg)
\int_{B_1}\frac{|u-u_*|^2}{|x|^2}\, dx\,.
\end{align*}
Choosing $\theta=1/9$ gives \eqref{eq:stab_sing_origin}.
Moreover, if equality occurs, then there must be equality in Hardy's inequality, which implies $f_\omega=0$ for a.e. $\omega$, and thus $u=u_*$.
\end{proof}

\subsection{Singularity away from the origin}

This section is devoted to the proof of Theorem~\ref{t:stabx0}. It relies on a generalization of Lemma~\ref{l:stab_sing_origin} where \(u_*\) is replaced by \(u_*(\cdot-x_0)\) for any \(x_0\in B_1\), 
see Lemma~\ref{l:stab_sing_x0}. We also need an estimate of $x_0$ in terms of the energy  deficit \(D(u)-D(u_*)\), whenever \(u\) is continuous except at \(x_0\). This inequality, which is encapsulated in Lemma~\ref{l:lower_bd_sing_x0}, is a particular case of more general estimates established in~\cite{BCL86}.

\begin{lem}\label{l:stab_sing_x0}
For any $x_0\in B_1$ and  $u\in H^1_{u_*}( B_1;\S^2)\cap C^0(\overline B_1\setminus \lbrace x_0\rbrace)$,  we have
\begin{align}\label{eq:stab_sing_x0}
\frac 19 \int_{B_1}|\nabla u-\nabla u_*(\cdot - x_0)|^2
\, dx
\leq \int_{B_1}|\nabla u|^2\, dx -8\pi +\frac{32\pi}{9}|x_0|^2.
\end{align}
\end{lem}

\begin{proof}[Proof of Lemma~\ref{l:stab_sing_x0}]
Consider $u$ as defined on $\R^2$ by setting $u=u_*$ outside $B_1$,
and define $\tilde u(\tilde x)=u(x_0+\tilde x)$,
which has degree one on $\partial B_r$ for a.e. $r>0$.
Arguing exactly as in Lemma~\ref{l:stab_sing_origin}, we obtain therefore
\begin{align*}
&
\int_{B_2}|\nabla \tilde u|^2\, dx - 16\pi
\\
&
\geq \theta\int_{B_2}|\nabla \tilde u-\nabla u_*|^2\, dx
+(1-\theta)\int_{B_2}|\partial_r \tilde u|^2\, dx - 2\theta\int_{B_2}\frac{|\tilde u-u_*|^2}{|x|^2}\, dx\,,
\end{align*}
for any $\theta\in (0,1)$.
Then we apply
the Hardy inequality
\begin{align*}
\frac 14 \int_0^2 |f(r)|^2\, dr \leq \int_0^2 |f'(r)|^2 r^2 dr 
+  |f(2)|^2\,,
\end{align*}
which follows from the identity
\begin{align*}
(f')^2 r^2 +\frac 12 (rf^2)' -\frac 14  f^2 =r\big((\sqrt r f)'\big)^2\geq 0\,,
\end{align*}
to $f(r)=f_\omega(r)=\tilde u(r\omega)-u_*(r\omega)$ for a.e. $\omega\in\mathbb S^2$, 
and integrate with respect to $\omega$,
to find
\begin{align*}
\int_{B_2}|\partial_r \tilde u|^2\, dx 
\geq \frac 14\int_{B_2}\frac{|\tilde u -u_*|^2}{|x|^2}\, dx 
- \frac 14 \int_{\partial B_2} |\tilde u - u_*|^2\, d\mathcal H^2\,.
\end{align*}
Plugging this into the previous estimate
and using that $\tilde u=u_*(\cdot +x_0)$ on $\partial B_2$
 gives
\begin{align*}
&
\int_{B_2}|\nabla \tilde u|^2\, dx - 16\pi
+\frac{1-\theta}{4}\int_{\partial B_2}|u_*(\cdot +x_0) -u_*|^2\, d\mathcal H^2
\\
&
\geq \theta\int_{B_2}|\nabla \tilde u-\nabla u_*|^2\, dx
+\frac{1-9\theta}{4}\int_{B_2}\frac{|\tilde u-u_*|^2}{|x|^2}\, dx\,.
\end{align*}
Choosing $\theta=1/9$
and replacing $\tilde u=u(\cdot +x_0)$, 
we infer
\begin{align}
\label{eq:stab_sing_x0_interm}
&
\frac{1}{9}\int_{B_2(x_0)}|\nabla u -\nabla u_*(\cdot -x_0)|^2\, dx
\nonumber
\\
&
\leq
\int_{B_2(x_0)}|\nabla u|^2-16\pi 
+ \frac{2}{9}\int_{\partial B_2}
|u_*(\cdot+x_0)-u_*|^2\, d\mathcal H^2.
\end{align}
To estimate the right-hand side, first note that
\begin{align}
\frac{1}{4}\int_{\partial B_2}
\!\!\!
|u_*(\cdot+x_0)-u_*|^2\, d\mathcal H^2
&
=\int_{\mathbb S^2} \left| \frac{\omega +x_0/2}{|\omega +x_0/2|}-\omega\right|^2\,d\mathcal H^2(\omega)
\nonumber
\\
&
\leq 4\pi |x_0|^2\,,
\label{eq550}
\end{align}
where we used the fact that,
 for any \(a,b\in \R^3\setminus \{0\}\),
\begin{align}\label{eq:ineqab}
\left|\frac{a}{|a|}-\frac{b}{|b|}\right|
= \frac{|(|b|-|a|)a +|a|(a-b)|}{|a| |b|} \leq 2 \frac{|a-b|}{|b|},
\end{align}
applied to \(a=\omega+x_0/2\) and \(b=\omega\).
Moreover,
since $u=u_*$  outside \(B_1\)
 and  \(B_1\subset B_2(x_0)\), we have
\begin{align*}
&\int_{B_2(x_0)}|\nabla u|^2 \, dx-16\pi 
-\left(\int_{B_1}|\nabla u|^2 \, dx - 8\pi\right)
\\
&= \int_{B_2(x_0)\setminus B_1}|\nabla u_*|^2\,dx-8\pi
=\int_{B_2(x_0)}|\nabla u_*|^2\, dx -\int_{B_2}|\nabla u_*|^2\, dx\,.
\end{align*}
Using the fact that 
the function 
$x_0\mapsto \int_{B_2(x_0)}|\nabla u_*|^2\,dx$ is radial and decreasing 
because
$|\nabla u_*|^2$ is radial and decreasing,
we thus get
\[
\int_{B_2(x_0)}|\nabla u|^2 \, dx-16\pi 
\leq
 \int_{B_1}|\nabla u|^2\, dx - 8\pi
\,.
\]
Plugging this inequality and \eqref{eq550} into~\eqref{eq:stab_sing_x0_interm}
 gives~\eqref{eq:stab_sing_x0}.
\end{proof}

\begin{lem}[{\cite{BCL86}}]
\label{l:lower_bd_sing_x0}
For any $x_0\in B_1$ and  $u\in H^1_{u_*}( B_1;\S^2)\cap C^0(\overline B_1\setminus \lbrace x_0\rbrace)$,  we have
\begin{align}
\label{eq:lower_bd_sing_x0}
\int_{B_1}|\nabla u|^2 \, dx - 8\pi \geq \frac{8\pi}{3} |x_0|^2.
\end{align}
\end{lem}

\begin{proof}[Proof of Lemma~\ref{l:lower_bd_sing_x0}]
This is a special case of estimates proved in \cite[\textsection{7}]{BCL86}.
We recall the proof for the reader's convenience.
We have
\begin{equation}\label{eq598}
\int_{B_1}|\nabla u|^2\, dx
\geq 2\int_{B_1}|D|\, dx,
\end{equation}
where
the vector field $D$ corresponds to the 2-form obtained by pulling back the volume form on $\mathbb S^2$ under the map $u$, and is explicitly given by
\begin{align*}
D=(u\cdot u_y\wedge u_z, u\cdot u_z\wedge u_x , u\cdot u_x\wedge u_y).
\end{align*}
If \(x,y\) are orthonormal local coordinates on \(\mathbb S^2\), then \(D\cdot n = u\cdot (u_x\wedge u_y)\) on \(\mathbb S^2\) (here, \(n\) is the  outer  unit normal to \(\mathbb S^2\)). Since \(n=u_*=u\) on \(\mathbb S^2\), it follows that  the normal component of $D$ is equal to $1$. 
Moreover, since \(u\)  has only one singularity of degree \(1\) at \(x_0\), one has 
$\dv D=4\pi\delta_{x_0}$  in $B_1$, see 
\cite[Appendix~B]{BCL86}.

Let us introduce the \(1\)-Lipschitz function $\zeta(x)=|x-x_0|$.
Since $|D|\geq |D\cdot\nabla \zeta|$, it follows from \eqref{eq598} that
\[
\int_{B_1}|\nabla u|^2\, dx 
\geq 2\bigg|\int_{B_1} D\cdot\nabla\zeta\, dx\bigg|\,.
\]
Integrating by parts in the right-hand side and using the fact that \(\dv D=4\pi\delta_{x_0}\) while \(\zeta(x_0)=0\),
we deduce
\[
\int_{B_1}|\nabla u|^2\, dx 
\geq 2\bigg|
\int_{\partial B_1}(D\cdot n)\zeta\,d\mathcal{H}^2 
\bigg|
= 2 \int_{\mathbb S^2} |\omega -x_0| \, d\mathcal H^2(\omega)\,.
\]
The right-hand side only depends on the norm of \(x_0\). Let us write
\begin{align*}
g:t\in (-1,1) \mapsto 2\int_{\mathbb S^2}|\omega -t e|\,d \mathcal{H}^{2}(\omega)
\end{align*}
for some fixed vector \(e\in \mathbb S^2\). Then, using spherical coordinates, one gets
\begin{align*}
g(t)=4\pi \int_{0}^{\pi}\sqrt{1+t^2-2t\cos\theta}\sin \theta \,d\theta=8\pi \left(1+\frac{t^2}{3}\right)\,,
\end{align*}
and \eqref{eq:lower_bd_sing_x0} follows.
\end{proof}

We now have all the ingredients to prove Theorem~\ref{t:stabx0}.

\begin{proof}[Proof of Theorem~\ref{t:stabx0}]
Plugging \eqref{eq:lower_bd_sing_x0} into \eqref{eq:stab_sing_x0} gives
\begin{align*}
\frac 19 \int_{B_1}|\nabla u-\nabla u_*(\cdot - x_0)|^2
\, dx
\leq \Big(1+\frac 43\Big)\bigg(\int_{B_1}|\nabla u|^2 -8\pi \bigg)\,,
\end{align*}
thus proving \eqref{eq:stabx0} with $c_*=1/21$.
For $x_0=0$ one can take $c_*=1/9$ thanks to 
 Lemma~\ref{l:stab_sing_origin}.
\end{proof}

\subsection{Necessity of the translation}\label{ss:translation}

In this section, we prove that the quadratic stability estimate~\eqref{eq:nondeg_stab} fails to be true in dimension 3. 
Our counterexample shows that taking translations in 
 Conjecture~\ref{conj:stab} seems unavoidable.

\begin{lem}\label{l:ex}
Fix an even  cut-off function 
$\eta\in C_c^\infty(B_1)$
such that
\begin{align*}
\mathbf 1_{|x|\leq 1/2}\leq \eta(x)\leq \mathbf 1_{|x|\leq 3/4}\,,
\end{align*}
and define, for  $|x_0|< 1/4$,
\begin{align*}
u^{x_0}(x)=u_*(x-\eta(x)x_0).  
\end{align*}
There exists a constant $C>0$ such that, for all $x_0\in B_{1/4}$,
\begin{align}\label{eq:ex}
C\int_{B_1}|\nabla u^{x_0}-\nabla u_*|^2\, dx 
&\geq |x_0| 
\quad
\text{and}
\quad
\int_{B_1}|\nabla u^{x_0}|^2\,dx -8\pi 
\leq C |x_0|^2.
\end{align}
\end{lem}
\begin{proof}
 Let \(x_0\in B_{1/4}\setminus \{0\}\).
 Using that $\eta=1$ on $B_{1/2}$ and $\nabla u_*$ is homogeneous of degree $-1$,
 and letting $\omega_0=x_0/|x_0|\in\mathbb S^2$,
 we find
\begin{align}
\int_{B_1}|\nabla u_*-\nabla u^{x_0}|^2\,dx &\geq \int_{B_{1/2}}|\nabla u_*(x) -\nabla u_*(x-x_0)|^2\,dx\,
\nonumber
\\
&=\frac{1}{|x_0|^2}\int_{B_{1/2}}\left|\nabla u_*\left(\frac{x}{|x_0|}\right) -\nabla u_*\left(\frac{x-x_0}{|x_0|}\right)\right|^2\,dx
\nonumber
\\
&
=|x_0|
\int_{B_{1/(2|x_0|)}}
|\nabla u_*(y)-\nabla u_*(y-\omega_0)|^2\, dy
\label{eq702}
\\
&
\geq 
|x_0|
\int_{B_{1}}
|\nabla u_*(y)-\nabla u_*(y-\omega_0)|^2\, dy\,.
\nonumber
\end{align}
The last integral is independent of $x_0$ by rotational invariance,
and positive since
 \(\nabla u_*(y)=\nabla u_*(z)\) if and only if \(y=\pm z\).
This completes the proof of the first inequality in~\eqref{eq:ex}.
 
The second inequality follows from the fact that
 the even function
\[
h:x_0\in B_{1/4}\mapsto \int_{B_1}|\nabla u^{x_0}|^2\,dx
\]
is smooth.
This can be checked by fixing 
 \(\chi\in C^{\infty}_c(B_{1/2})\) such that \(\chi\equiv 1\) on \(B_{1/3}\), so that we have
\begin{align*}
h(x_0)
=\int_{\R^3}\chi(x)|\nabla u_*(x-x_0)|^2\,dx 
+ \int_{B_1\setminus B_{1/3}}(1-\chi(x))|\nabla u^{x_0}|^2\,dx\,.
\end{align*} 
The first integral 
depends smoothly on $x_0$ as the convolution of $\chi\in C_c^\infty(B_{1/2})$ with $|\nabla u_*|^2\in L^1_{\loc}(\R^3)$.
The second integral depends smoothly on $x_0\in B_{1/4}$ 
because $u_*$ is smooth on $\R^3\setminus 0$ and 
$|x-\eta(x)x_0|\geq 1/12$ for all $x\in B_{1}\setminus B_{1/3}$ and $x_0\in B_{1/4}$.
\end{proof}

\begin{rem}\label{r:ux0higherdim}
In dimension $n\geq 4$,
 consider $u^{x_0}$ defined as in Lemma~\ref{l:ex} in $B_1\subset\R^n$.
For every $x_0\in B_{1/4}\setminus\lbrace 0\rbrace$, 
letting $\omega_0=x_0/|x_0|$
 we find, as in \eqref{eq702},
\begin{equation}\label{eq869}
\int_{B_{\frac 12}}|\nabla u_* -\nabla u^{x_0}|^2\,dx
=|x_0|^{n-2}\int_{B_{\frac{1}{2|x_0|}}}
|\nabla u_*(y)-\nabla u_*(y-\omega_0)|^2
\,dy\,.
\end{equation}
We denote by $C=C(n)>0$ a generic dimensional constant 
whose precise value may change from line to line in what follows.
Direct calculation shows that
\begin{align*}
\frac{1}{C}\leq |\nabla^2 u_*(z)\cdot \omega_0|\leq C\,,\qquad\forall z\in\S^{n-1}\,.
\end{align*}
The homogeneity of $\nabla u_*$ and a contradiction argument therefore imply
\begin{align*}
\frac 1C \leq |y|^2 |\nabla u_*(y)-\nabla u_*(y-\omega_0)| \leq C\,,
\qquad\forall y\in \R^3\setminus B_2\,.
\end{align*}
Using also that 
\begin{align*}
\int_{B_2}|\nabla u_*(y)-\nabla u_*(y -\omega_0)|^2\, dy\leq 
4\int_{B_3}|\nabla u_*|^2\, dy\leq C\,,
\end{align*} 
we deduce
\begin{align*}
\frac 1C \int_{2}^{\frac{1}{2|x_0|}}
\!\!
r^{n-5}\, dr
&
\leq 
\int_{B_{\frac{1}{2|x_0|}}\setminus B_{2}}
\!\!\!
|\nabla u_*(y)-\nabla u_*(y-\omega_0)|^2
\,dy
\\
&
\leq 
\int_{B_{\frac{1}{2|x_0|}}}
|\nabla u_*(y)-\nabla u_*(y-\omega_0)|^2
\,dy
\\
&
\leq C +  C\int_{2}^{\frac{1}{2|x_0|}}
\!\!
r^{n-5}\, dr
\,.
\end{align*}
Combining these inequalities with \eqref{eq869}
and the fact that the function 
\begin{align*}
x_0\mapsto \int_{B_{1}\setminus B_{1/2}}|\nabla u_* -\nabla u^{x_0}|^2\,dx
\end{align*}
is smooth  on $B_{1/4}$ and minimal at $x_0=0$,
we deduce the existence of $\Theta_n(x_0)\in [C^{-1},C]$ such that
\begin{align*}
\int_{B_1}|\nabla u_* -\nabla u^{x_0}|^2\,dx
= 
\Theta_n(x_0)\cdot
\begin{cases}
|x_0|^2\ln\frac 1{|x_0|} &\quad\text{if }n=4\,,
\\
|x_0|^2  &\quad\text{ if }n\geq 5\,,
\end{cases}
\end{align*}
for all $x_0\in B_{1/4}$.
For the energy deficit, the same argument as in Lemma~\ref{l:ex} provides the upper bound
 $D(u^{x_0})-D(u_*)\leq C |x_0|^2$.
For $n=4$, this shows that $D(u^{x_0}-u_*)\gg D(u^{x_0})-D(u_*)$ as $|x_0|\to 0$.

\end{rem}

\section{Minimality for the Oseen-Frank energy}\label{s:minOF}

In this section,  we prove Theorem~\ref{t:minOF} and Theorem~\ref{t:minOF0}; that is,  the minimality of $u_*$ for some values of $\mathbf k$ beyond the known regime $k_1\leq k_2$.

\subsection{Link with $\curl$ stability estimates}

The proofs of Theorem~\ref{t:minOF} and Theorem~\ref{t:minOF0} crucially rely on 
a quantitative control of $\curl u$ in terms of the Dirichlet energy deficit \(D(u)-D(u_*)\):

\begin{lem}\label{l:stabcurlOF}
If $u\in H^1_{u_*}(B_1;\mathbb S^2)$ satisfies
\begin{align}\label{eq:assume_stab_c}
c \int_{B_1}|\curl u|^2\, dx  \leq \int_{B_1}|\nabla u|^2 -  \int_{B_1}|\nabla u_*|^2\,,
\end{align}
for some $c\in (0,1)$,
then, for  any $\mathbf k =(k_1,k_2,k_3)$
such that
\begin{equation}\label{eq837}
0<k_1,k_2,k_3,
\quad k_1 \leq k_2+ \frac{c}{1-c} \min(k_2,k_3),
\end{equation}
we have
\begin{align*}
E_{\mathbf k}(u) \geq E_{\mathbf k}(u_*)\,.
\end{align*}
If equality occurs, then 
equality holds in~\eqref{eq:assume_stab_c},  
and either  $u=u_*$,
 or $k_1=k_2+  ck_3/(1-c)$.
\end{lem}

\begin{proof}
We write
\begin{align*}
E_{\mathbf k}(u)=k_1I_1+k_2I_2+k_3I_3\,,
\end{align*}
where
\begin{align*}
I_1=\int_{B_1}(\dv u)^2\, dx \,,
\quad
I_2=\int_{B_1}(u\cdot \curl u)^2\, dx\,,
\quad
I_3=\int_{B_1}|u\times\curl u|^2\, dx\,.
\end{align*}
By \cite[Theorem 1]{ou92}, if \(k_2\geq k_1\), then \(E_{\mathbf k}(u)\geq E_{\mathbf k}(u_*)\) with equality if and only if \(u=u_*\). In particular, equality holds in~\eqref{eq:assume_stab_c}.

We may therefore assume from now on that $k_1 > k_2$. 
In view of~\eqref{eq171}, 
\begin{align*}
I_1+I_2+I_3 =E_{(1,1,1)}(u)=\int_{B_1}|\nabla u|^2\, dx
+  8\pi\,,
\end{align*}
and thus,  our assumption~\eqref{eq:assume_stab_c} can be rewritten as
\begin{align*}
c(I_2+I_3)\leq I_1+I_2+I_3 -16\pi\,,
\end{align*}
or equivalently
\begin{align}\label{eq:stab_I}
I_1+(1-c)I_2+(1-c)I_3 \geq 16\pi\,.
\end{align}
By~\cite[Proposition~4]{ou92},  we have
\begin{equation}\label{eq:ineqLin_I}
I_1+I_2 =E_{(1,1,0)}(u) \geq E_{(1,1,0)}(u_*)=\int_{B_1}(\dv u_*)^2\,dx =16\pi\,.
\end{equation}
For any $t\in [0,1]$, 
multiplying \eqref{eq:stab_I} by  $t k_1$ and \eqref{eq:ineqLin_I} by  $(1-t)k_1$, 
we find
\begin{align}\label{eq:ItEk}
k_1 I_1 +(1-tc)k_1 I_2 + t(1-c)k_1 I_3 \geq 16\pi\, k_1 = E_{\mathbf k}(u_*)\,.
\end{align}
By \eqref{eq837}, one has
\[
(1-c)k_1\leq (1-c)k_2+ck_2=k_2\,,
\]
and thus \(k_1-k_2\leq ck_1\).
Together with the fact that $k_2< k_1$, this entitles one to choose
\begin{align*}
t=\frac{k_1-k_2}{ck_1}\in (0,1]\,.
\end{align*}
Observe that \((1-tc)k_1=k_2\) and also that
\[
t(1-c)k_1=\frac{k_1-k_2}{\delta}\,,
\]
where we have set \(\delta=c/(1-c)\).
It thus follows from~\eqref{eq:ItEk} that
\begin{align*}
k_1 I_1 +k_2 I_2 + \frac{k_1-k_2}{\delta} I_3 \geq E_{\mathbf k}(u_*)\,,
\end{align*}
and therefore
\begin{align*}
k_1I_1+k_2 I_2+k_3 I_3
=
E_{\mathbf k}(u)
 \geq E_{\mathbf k}(u_*) +\frac{\delta k_3 +k_2 - k_1}{\delta}I_3\,.
\end{align*}
By~\eqref{eq837} again, $k_1\leq k_2+\delta k_3$, so that
 $E_{\mathbf k}(u)\geq E_{\mathbf k}(u_*)$.
 
Assume that equality occurs; that is, \(u\) minimizes \(E_{\mathbf k}\). 
Since $t>0$ we must have equality in \eqref{eq:stab_I};
that is, equality in the assumption \eqref{eq:assume_stab_c}.
We must also have either $k_1=k_2+\delta k_3$ or  $I_3=0$. 
By the regularity properties of the minimizers of \(E_{\mathbf k}\) developed in \cite{HKL86, HKL88}, the singular set of \(u\) is closed and has zero Lebesgue measure. We can thus repeat the proof of \cite[Theorem 1]{ou92} showing that if \(I_3=0\), then \(u=u_*\).
\end{proof}

\subsection{Proof of Theorem~\ref{t:minOF0}}

From Lemma~\ref{l:stabcurlOF} and the stability estimate of Lemma~\ref{l:stab_sing_origin} for maps which are continuous away from the origin, we can readily prove Theorem~\ref{t:minOF0}.

\begin{proof}[Proof of Theorem~\ref{t:minOF0}]
Let $u\in H^1_{u_*}(B_1;\mathbb S^2)$. 
We  introduce the map \(v=u-u_*\).
Since $\curl u_*=0$, one has
\[
\int_{B_1}|\curl u|^2\, dx 
=\int_{B_1}|\curl v|^2\, dx \leq \int_{B_1} \bigg((\dv v)^2+|\curl v|^2 \bigg)\, dx\,.
\]
Using that $v=u-u_*$ vanishes on $\partial B_1$  and that $|\nabla v|^2-|\curl v|^2-(\dv v)^2$ is a null Lagrangian, we thus get
\[
\int_{B_1}|\curl u|^2\, dx 
\leq \int_{B_1}|\nabla v|^2\, dx = \int_{B_1}|\nabla u-\nabla u_*|^2\, dx\,.
\]
Assuming in addition that $u$ is continuous in $ \overline B_1 \setminus\lbrace 0\rbrace$, we can apply Lemma~\ref{l:stab_sing_origin} to deduce that
\begin{align}\label{eq:stabcurl0}
\frac 19\int_{B_1}|\curl u|^2\, dx 
\leq 
\int_{B_1}|\nabla u|^2\, dx - 8\pi\,.
\end{align}
Lemma~\ref{l:stabcurlOF} therefore implies that $E_{\mathbf k}(u)\geq E_{\mathbf k}(u_*)$ if $k_1\leq k_2 +\min(k_2,k_3)/8$.
Moreover, equality in this inequality implies equality in  \eqref{eq:stab_sing_origin}, hence $u=u_*$.
\end{proof}

\subsection{Proof of Theorem~\ref{t:minOF}}

The proof of Theorem~\ref{t:minOF} follows  the same strategy, 
but we first need to invoke an argument in \cite{AL88} showing that for $\mathbf k\approx \mathbf 1$, the 
minimizers of $E_{\mathbf k}$ have a single singularity.
For the readers' convenience, we provide a sketch of the proof.

\begin{lem}[{\cite[\textsection{6.1}]{AL88}}]
\label{l:minOFsing}
There exists $\eta>0$ such that,
 if $k_1,k_2,k_3>0$ satisfy $|k_2/k_1-1|+|k_3/k_1-1|<\eta$, then any minimizer $E_{\mathbf k}$ 
 in $H^1_{u_*}(B_1;\mathbb S^2)$
is smooth away from a single point $x_0\in B_1$.
\end{lem}
\begin{proof}
Dividing $E_{\mathbf k}$ by $k_1$,  we assume without loss of generality that $k_1=1$.
Then we argue by contradiction and assume the existence of 
 a
sequence of triples
 $\mathbf k^{(j)}=(1,k_{2}^{(j)},k_3^{(j)})\to \mathbf 1$ 
and maps $u^{(j)}$ 
minimizing 
$E_{\mathbf k^{(j)}}$ in $H^1_{u_*}(B_1;\mathbb S^2)$ such that $u^{(j)}$ has at least two singular points $x^{(j)}\neq y^{(j)}\in B_1$.
We extract a weak limit $u^{(j)}\rightharpoonup \bar u$ in $H^1(B_1;\R^3)$. Then $\bar u\in H^1_{u_*}(B_1;\S^2)$. 
By weak lower semicontinuity of \(E_{\mathbf 1}\) and minimality of each \(u^{(j)}\), this implies that
\[
E_{\mathbf{1}}(\bar u)\leq \liminf E_{\mathbf k^{(j)}}(u^{(j)})
\leq \limsup E_{\mathbf k^{(j)}}(u^{(j)})\leq 
\limsup E_{\mathbf k^{(j)}}(u_*)=E_{\mathbf{1}}(u_*)\,.
\]
Hence, \(\bar u=u_*\) by uniqueness of the minimizer of  \(D=E_{\mathbf{1}}-8\pi\), see \cite[Theorem 7.1]{BCL86}. Moreover, all the above inequalities are equalities. 
This implies that $u^{(j)}\to u_*$ strongly in $H^1$, see e.g. \cite[Lemma 5.2]{HKL88}.
We next rely on the small energy regularity 
for minimizers of \(E_{\mathbf{k}}\) developed in \cite[Section 2]{HKL86}: given a compact subset \(K\subset (0,\infty)^3\), 
there exists \(\varepsilon>0\) such that for every \(a\in B_1\), for every \(R>0\) such that \(B_{2R}(a)\subset B_1\), for every \(\mathbf{k}\in K\), if \(w\) is a  minimizer of \(E_{\mathbf{k}}\) on \(B_1\) such that
\[
\frac{1}{R}\int_{B_{2R}(a)}|\nabla w|^2\leq \varepsilon,
\] 
then \(w\) is smooth in \(B_{R}(a)\). 
We proceed to apply this statement for \(K=\lbrace\mathbf{k}^{(j)}\rbrace_{j\geq 1}\cup \{\mathbf{1}\}\).
Using the fact that $u_*$ is smooth away from \(0\) and the strong convergence of \(u^{(j)}\) to \(u_*\), 
the above  small energy regularity result implies that for every \(\delta\in (0,1)\),  there exists \(j_{\delta}\geq 1\) such that 
$u^{(j)}$ is smooth away from $B_\delta$ for every $j\geq j_\delta$.
As a consequence, the singularities $x^{(j)}\neq y^{(j)}$ must both converge to $0$.
Now consider the rescaled maps
\begin{align*}
\hat u^{(j)}(\hat x)=u^{(j)}(x^{(j)}+r^{(j)}\hat x),\quad r^{(j)}=|x^{(j)}-y^{(j)}|\to 0\,.
\end{align*}
For any $R>0$, we have $x^{(j)}+r^{(j)} B_{R}\subset B_1$ if $j$ is large enough, and $\hat u^{(j)}$ is a
 minimizer of $E_{\mathbf k^{(j)}}$ among maps in $H^1(B_{R};\S^2)$ with the same boundary values.
The compactness result   \cite[Theorem 5.3]{HKL88} \footnote{This result in \cite{HKL88} is stated for a single functional \(E_{\mathbf k}\) but the proof generalizes at once for a family of functionals \(E_{\mathbf k^{(j)}}\) provided that the parameters 
\(\mathbf k^{(j)}\) remain in a compact subset of \((0,\infty)^3\), as this is the case in our situation.}
 and a diagonal argument imply therefore that, 
 after extracting a subsequence, $\hat u^{(j)}\to\hat u$ strongly in $H^1_{loc}(\R^3;\R^3)$, and $\hat u\in H^1_{\loc}(\R^3;\S^2)$ is a local minimizer of the Dirichlet energy $D$.
According to \cite[Theorem~2.2]{AL88}, 
the map $\hat u$ is smooth away from at most one singular point $\hat z$.
By strong convergence and using again the small energy regularity for the minimizers of $E_{\mathbf k}$, this implies that $\hat u^{(j)}$ is smooth in $B_2(0)\setminus B_{\delta}(\hat z)$ for any small $\delta >0$ and large enough $j$,
in contradiction with the fact that $\hat u^{(j)}$ has a singularity at 0 and another singularity on $\partial B_1$.
\end{proof}

\begin{proof}[Proof of Theorem~\ref{t:minOF}]
By Lemma~\ref{l:minOFsing}, there exists $\eta>0$ such that, 
if $|k_2/k_1-1|+|k_3/k_1-1|\leq \eta$, 
then any minimizer $u$ of $E_{\mathbf k}$ in $H^1_{u_*}(B_1;\mathbb S^2)$ has a single singularity $x_0\in B_1$.
Setting $v=u-u_*(\cdot - x_0)$,
and using that $\curl u_*=0$ and $|\curl v|^2\leq 2|\nabla v|^2$,
we find
\[
\int_{B_1}|\curl u|^2\, dx
=\int_{B_1}|\curl v|^2\, dx
\leq 2\int_{B_1}|\nabla v|^2\, dx =2\int_{B_1}|\nabla u-\nabla u_*(\cdot -x_0)|^2\,dx\,.
\]
Combining this with the stability estimate of Theorem~\ref{t:stabx0} we deduce
\begin{align*}
\frac{c_*}{2}\int_{B_1}|\curl u|^2\, dx 
\leq \int_{B_1}|\nabla u|^2\,dx -8\pi\,.
\end{align*}
According to Lemma~\ref{l:stabcurlOF} 
this implies that $u_*$ is the unique minimizer of $E_{\mathbf k}$ if
$k_1 < k_2+\delta\min(k_2,k_3)$, with $\delta =c_*/(2-c_*)$.
Here, the parameters \(k_1, k_2\) and \(k_3\) are also subject to the additional restriction
$|k_2/k_1-1|+|k_3/k_1-1|\leq \eta$.
We deduce in particular that for any \(\eta'>0\) sufficiently small, the map $u_*$ is the unique minimizer of $E_{\mathbf k}$ for 
$\mathbf k=(1+\eta',1,1)$. We can thus repeat the calculation presented in Remark~\ref{r:curl} to infer the validity of \eqref{eq:assume_stab_c} for $c=\eta'/(1+\eta')>0$ and all $u\in H^1_{u_*}(B_1;\S^2)$.
The conclusion thus follows from Lemma~\ref{l:stabcurlOF}.
\end{proof}

\bibliographystyle{acm}
\bibliography{oseenfrank}

\end{document}